\documentclass[a4paper,12pt]{amsart}
\usepackage{amsmath}
\usepackage{amsfonts}
\usepackage{amssymb}
\usepackage{setspace}
\usepackage{amsthm,mathrsfs,textcomp}

\newtheorem*{thA}{Theorem A}
\newtheorem*{thB}{Theorem B}

\theoremstyle{plain}
\newtheorem{theorem}{Theorem}[section]
\newtheorem{cor}[theorem]{Corollary}
\newtheorem{prop}[theorem]{Proposition}

\newtheorem{lemma}[theorem]{Lemma}

\theoremstyle{definition}
\newtheorem{remark}[theorem]{Remark}

\newtheorem{definition}[theorem]{Definition}

\begin{document}

\title[Nonstandard Nullstellensatz]{Arithmetic Nullstellensatz and Nonstandard Methods}
\author{Haydar G\"oral}
\address{Universit\'e de Lyon CNRS, Universit\'e Lyon 1,Institut
Camille Jordan UMR5208, 43 boulevard du 11 novembre 1918, F--69622
Villeurbanne Cedex, France.}
\email{goral@math.univ-lyon1.fr}


\begin{abstract}
In this study we find height bounds
for polynomial rings over integral domains.
We apply nonstandard methods and hence our constants will be ineffective.
Then we find height bounds in the polynomial ring over algebraic numbers 
to test primality of an ideal.
Furthermore we consider unique factorization domains and
possible bounds for valuation rings and arithmetical functions.
\end{abstract}

\maketitle

\onehalfspacing

\section{Introduction}
The arithmetic version of the Nullstellensatz states that 
if $f_1,...,f_s$ belong to $\mathbb Z[X_1,...,X_n]$ without a common zero in $\mathbb C$, then
there exist $a$ in $\mathbb Z \setminus \{0\}$ and
$g_1,...,g_s$ in $\mathbb Z[X_1,...,X_n]$
such that $a=f_1g_1+...+f_sg_s.$ Finding degree and height bounds for $a$ and 
$g_1,...,g_s$ has received continuous attention using computational methods.
By $\deg f$, we mean the total degree of  the polynomial $f$ in several variables.
T. Krick, L. M. Pardo and
M. Sombra \cite{KPS} prove that:
If $f_1,...,f_s$ 
are as above with
$D:=\displaystyle\max_{i} \deg (f_i)$ and $H:=\displaystyle\max_{i} h(f_i)$
where $h(f_i)=$ logarithm of the maximum module of its coefficients, then
there exist $a \in \mathbb Z \setminus \{0\}$ and $g_1,...,g_s \in \mathbb Z[X_1,...,X_n]$
such that
\begin{itemize}
\item[(i)] $a=f_1g_1+...+f_sg_s$
\item[(ii)] $\deg(g_i) \leq 4nD^n$
\item[(iii)] $h(a),h(g_i) \leq 4n(n+1)D^n(H+\log s + (n+7)\log (n+1)D).$ 
\end{itemize}

This result is sharp and efficient.
For similar results we refer the reader to \cite{Ber, BS}.

On the other hand finding bounds in mathematics using nonstandard extensions
have been studied often, for example:
Given a field $K$,
if $f_0,f_1,...,f_s$ in $K[X_1,...,X_n]$ all have degree less than $D$ and
$f_0$ in $\langle f_1,...,f_s \rangle$, then
$f_0=\displaystyle\sum_{i=1}^{s}f_ih_i$ for certain ${h_i}$ whose degrees are bounded
by a constant $C=C(n,D)$ depending only on $n$ and $D$.
This result was first established in a paper of G. Hermann \cite{Her} using algorithmic tools.
Then the same result was proved by L. van den Dries and K. Schmidt \cite{DS} using
nonstandard methods, and they paved the way for how nonstandard methods can be used
for such bounds. Their work in \cite{DS} influenced us
to apply nonstandard methods in order to prove the existence of
bounds for the complexity of the coefficients of $h_i$ as above by taking $f_0=1.$
We also define an abstract height function which generalizes
the absolute value function
and measures the complexity of the coefficients of polynomials over 
$R[X_1,...X_n]$, where $R$ is an integral domain.
Using nonstandard methods, we will generalize the result of  \cite{KPS}
to any integral domain and height function and furthermore our constant $c_2$
for the height function does not depend on $R$ or $s$,
but it is ineffective.
\\

Let $K$ be a field and $I$ an ideal of $K[X_1,...,X_n].$
We say that $I$ is a $D$-type ideal if the degree of all the generators of $I$ is bounded by $D.$
By \cite{DS} it is known that there is a bound $B(n,D)$
such that if $I$ is a $D$-type ideal then $I$ is prime
iff $1 \notin I$, and for all $f$, $g$ in $K[X_1,...,X_n]$ of degree less than $B(n,D)$, if 
$fg \in I$ implies $f$ or $g$ is in $I.$
Here we show that it is enough to check the primality up to a certain height bound.

Let $\overline{\mathbb Q}$ be the set of algebraic numbers.
We say that an ideal $I$ of $\overline{\mathbb Q}[X_1,...,X_n]$ is a $(D,H)$-type ideal if it
is a $D$-type ideal and the logarithmic height of all generators of $I$ is bounded by $H.$
\\

We assume that all rings are commutative with unity.
Moreover throughout this article $R$ stands
for an integral domain and $K$ for its field of fractions.
The symbol $h=h_R$ denotes a height function on $R$ which will be defined in the next section.
We prove the following theorems:
\begin{thA} Let $R$ be a ring with a height function.
For all $n \geq 1$, $D \geq 1$, $H \geq 1$ there are two constants $c_1(n,D)$ and $c_2(n,D,H)$
such that if $f_1,...,f_s$ in $R[X_1,...X_n]$ have no common zero in $K^{alg}$ with
$\deg(f_i) \leq D$ and $h(f_i) \leq H$, then there exist nonzero $a$ in $R$ and
$h_1,...,h_s$ in $R[X_1,...X_n]$ such that
\begin{itemize}
\item[(i)] $a=f_1h_1+...+f_sh_s$
\item[(ii)] $\deg(h_i) \leq c_1$
\item[(iii)] $h(a),h(h_i) \leq c_2.$
\end{itemize}

\end{thA}

\begin{thB}
Let $h$ be the logarithmic height function.
There are bounds $B(n,D)$ and $C(n,D,H)$ such that if $I$
is a $(D,H)$-type ideal of $\overline{\mathbb Q}[X_1,...,X_n]$ then $I$ is prime
iff $1 \notin I$, and for all $f$, $g$ in $\overline{\mathbb Q}[X_1,...,X_n]$
of degree less than $B(n,D)$ and height less
than $C(n,D,H)$,
$fg \in I$ implies $f$ or $g$ is in $I.$
\end{thB}
\section{Preliminaries}

\subsection{Height Function}
Let $\theta:{\mathbb N} \rightarrow \mathbb N$ be a function.
We say that
$$h:R \rightarrow [0,\infty)$$
is a height function of $\theta$-type if
for any $x$ and $y$ in $R$ with $h(x) \leq n$ and $h(y) \leq n$,
then both $h(x+y)\leq \theta(n)$
and $h(xy) \leq \theta(n)$.
We say that $h$ is a height function on $R$ if $h$ is
a height function of $\theta$-type for some
$\theta:{\mathbb N} \rightarrow \mathbb N.$

We can extend the height function $h$ to the polynomial ring $R[X_1,...X_n]$ by
$$
h\bigg(\displaystyle\sum_{\alpha}a_{\alpha}X^{\alpha}\bigg)
=\displaystyle\max_{\alpha}h(a_{\alpha}).
$$
Note that this extension does not have to be a height function,
it is just an extension of functions.
Now we give some examples of height functions.

\textbf{Examples:}
For the following examples of height functions, one can take
$\theta(n)=(n+1)^2$.
\begin{itemize}
\item If $(R,|\cdot|)$ is an absolute valued ring 
then $h(x)=|x|$ is a height function.
Moreover $h(x)=|x|+1$ and $h(x)=\max(1,|x|)$ 
are also height functions on $R.$
\item The degree function on 
$R[X_1,...,X_n]$ is a height function.
\item Let $\lambda$ be a positive real number.
On ${\mathbb Z}[X]$, define
$$h(a_0+a_1X+...+a_kX^k)=\displaystyle\sum_{i=0}^{k} |a_i|{\lambda}^i.$$
Then this is a height function on ${\mathbb Z}[X].$
\item Let $h:R \rightarrow [0,\infty)$ be a function such that
the sets 
$$A_n=\{x \in R: h(x) \leq n\}$$
are all finite for all $n \geq 1.$
Then $h$ is a height function of $\theta$-type where 
$\theta(n)=\displaystyle\max_{x,y \in A_n}\{h(x+y)+h(xy)\}.$
\item The $p$-adic valuation on $\mathbb Z$ is not a height function.
Note that
$1$ and $p^n - 1$ are not divisible by $p$ but
but their sum is divisible by $p^n.$

\end{itemize}


\subsection{The Logarithmic Height Function}
For the details of this subsection we refer the reader to 
\cite{Bom, Eisenbud, HS, Lang}.
\\

For  $f(x)=a_d(X-{\alpha}_1)...(X-{\alpha}_d) \in {\mathbb C}[X]$
the Mahler measure of $f$ is defined as 
$$M(f)=\displaystyle|a_d|\prod_{|\alpha_j| \geq 1}|\alpha_j|.$$
For $\alpha$ in $\overline{\mathbb Q}$ with minimal polynomial $f(x) \in \mathbb Z[X]$,
we define its Mahler measure as 
$M(\alpha)=M(f).$
The absolute non-logarithmic height of $\alpha$ is defined as
$$H(\alpha)= M(\alpha)^{1/d}.$$
Then the logarithmic height of $\alpha$ is defined as
$$h(\alpha)=\log H(\alpha)=\frac{\log M(\alpha)}{d}.$$
It is not known whether there exists an absolute constant $c>1$ such that
if $M(\alpha)>1$ then $M(\alpha) \geq c$.
This question was posed by D. Lehmer \cite{Leh} around 1933.
The best known example of the smallest Mahler measure
greater than 1 so far was also given by Lehmer:
if $\alpha$ is a root of the polynomial
$$X^{10} + X^9 - X^7 - X^6 - X^5 - X^4 - X^3 + X + 1$$
then $M(\alpha) \approx 1.17628.$
For detailed results on Mahler measure and Lehmer's problem, see \cite{S}.

The logarithmic height function is a function
that measures the complexity of an algebraic number.
The logarithmic height function behaves well
under arithmetic operations but using this definition it is not
immediate to see. So we will give an equivalent definition using absolute values.

Let $K$ be a number field containing $\alpha.$
We define the relative height 
$$H_K(\alpha)=\displaystyle\prod_{v \in M_K} max\{1,||\alpha||_v\}$$
where $M_K$ is a set of absolute values extending the absolute values on $\mathbb Q$,
satisfying the product formula with multiplicities $N_v=[K_v:{\mathbb Q}_v]$
and $||\alpha||_v={|\alpha|_v}^{N_v}.$

Then absolute non-logarithmic height becomes $${H_K(\alpha)}^{1/[K:{\mathbb Q}]}$$
and this does not depend the choice of $K.$
Now one can see the height function behaves well under arithmetic operations:

\begin{itemize}
\item $H(\alpha+\beta)\leq 2H(\alpha)H(\beta)$
\item $H(\alpha\beta) \leq H(\alpha)H(\beta)$
\item $H(1/\alpha)=H(\alpha)$

\end{itemize}

\begin{lemma} \label{mahler}
Suppose $f=a_0+...+a_dX^d \in {\mathbb C}[X].$
Put $|f|=\max_{i}\{|a_i|\}.$
Then $2^{-d}|f| \leq M(f) \leq 2^{2d+1}|f|.$
\end{lemma}

Now we give the Gauss lemma.
First put $|f|_v=\max_{i}\{|a_i|_v\}.$
\begin{lemma} \label{gauss}
Let $K$ be a number field and 
suppose $f$, $g$ are in $K[X].$
For a non-archimedean absolute value $v$ on $K$, we have
$|fg|_v=|f|_v|g|_v$
\end{lemma}

\subsection{Height inequality}
There is a relation between height of a polynomial and height of its roots.
Define $H(f)=\max_{i} H(a_i)$ as before.
Then if $f$ is a polynomial over a number field $K$, we have 
$$H_K(f)=\displaystyle\prod_{v \in M_K} max\{1,||f||_v\}$$
and $H(f)={H_K(f)}^{1/[K: \mathbb Q]}.$

\begin{lemma} \label{heightt}
For  $$f(x)=(X-{\alpha}_1)...(X-{\alpha}_d)=a_0+...+X^d \in \overline{\mathbb Q}[X],$$
$H({\alpha}_i)$ is uniformly bounded by $H(f)$ and $d$ i.e
$$2^{-d}H(f) \leq \prod_{i} H({\alpha}_i) \leq 2^{2d+1}H(f).$$
\end{lemma}

\begin{proof}
Let $K$ be a number field containing ${\alpha}_i,a_j.$
By \eqref{mahler} we see that $ 2^{-d}|f| \leq M(f) \leq 2^{2d+1}|f|.$
For non-archimedean $v \in M_K$, by \eqref{gauss} we see that
$|f|_v=\prod_{i \leq d}\max\{1, |{\alpha}_i|_v \}$.
Therefore since $M(f) \geq 1$, we obtain that
$$ 2^{-d}\prod_{v}\max\{1,|f|_v\} \leq \prod_{i,v}\max\{1,|{\alpha}_i|_v\} \leq 2^{2d+1}\prod_{v}\max\{1,|f|_v\}.$$
Hence we get 
$$ 2^{-d}H(f) \leq \prod_{i} H({\alpha}_i) \leq 2^{2d+1}H(f).$$

\end{proof}

\subsection{Nonstandard Extensions and Height Function}
Now we define a nonstandard extension following
\cite{Hen}.

\begin{definition}[Nonstandard Extension of a Set] Let $\mathbb M$ be a nonempty
set. A nonstandard extension of $\mathbb M$ consists of a mapping that
assigns a set $^*A$ to each $A$ in ${\mathbb M}^m$ for all $m \geq 0$,
such that $^*{\mathbb M}$ is non-empty
and the following conditions are satisfied for all $m,n \geq 0$:
\\

(E1) The mapping preserves Boolean operations on subsets of ${\mathbb M}^m$:
if $A \subseteq {\mathbb M}^m$, then $^*A \subseteq (^*{\mathbb M})^m$;
if $A;B \subseteq {\mathbb M}^m$, then
$^*(A \cup B) = {^*A} \cup {^*B}$, $^*(A \cap B) = {^*A} \cap {^*B}$ and
$^*(A \setminus B) = {^*A} \setminus {^*B}.$

(E2) The mapping preserves basic diagonals:
if $1 \leq i < j \leq m$ and $\Delta = \{(x_1,..., x_m) \in {\mathbb M}^m:
x_i = x_j \}$ then
$^*{\Delta} = \{(x_1,..., x_m) \in (^*{\mathbb M})^m:
x_i = x_j \}.$

(E3) The mapping preserves Cartesian products:
if $A \subseteq {\mathbb M}^m$ and $B \subseteq {\mathbb M}^n$,
then $^*(A \times B) = {^*A} \times {^*B}.$ (We regard $A \times B$ as a
subset of ${\mathbb M}^{m+n}$.)

(E4) The mapping preserves projections that omit the final coordinate:
let $\pi$ denote projection of $n+1$-tuples on the first $n$ coordinates;
if $A \subseteq {\mathbb M}^{n+1}$
then $^*(\pi(A)) = \pi(^*A).$
\end{definition}

The set $^*{\mathbb M}$ will denote the nonstandard extension of $\mathbb M$.
For example, an ultrapower of ${\mathbb M}$ which
respect to a nonprincipal ultrafilter on $\mathbb N$ is
a proper nonstandard extension of $\mathbb M.$
Subsets of $^*{\mathbb M}$ of the form $^*A$ for some
subset $A$ of ${\mathbb M}$ are called internal. Not every subset
of $^*{\mathbb M}$ need to be internal.
We list the basic properties of nonstandard extensions with no proof.
\begin{itemize}

\item For each $n \geq 0$, $^*({\mathbb M}^n)=(^*{\mathbb M})^n$ and $^*{\emptyset}=\emptyset.$
\item For any $A,B \subseteq {\mathbb M}^n$, $^*A={^*B}$ iff $A=B.$
\item For each $x \in {\mathbb M}$, the set $^*{\{x\}}$ has exactly one element.
\\
Using the properties above, we can embed ${\mathbb M}$ into $^*{\mathbb M}.$
So without loss of generality we may assume that ${\mathbb M}$
is a subset of $^*{\mathbb M}.$ Moreover, if $A \subseteq {\mathbb M}^n$ then
${^*A} \cap {\mathbb M}^n=A^n$, in particular, $A \subseteq {^*A}.$
Also every function on $A$ extends to a function on $^*A.$
The new function is denoted by $^*f$, but without confusion we write $f$
instead. Lastly we give the most important property of nonstandard extensions.
\item \textbf{Transfer formula:} The two sets ${\mathbb M}$ and $^*{\mathbb M}$ satisfy
the same first order sentences. Moreover if $\phi(v_1,...,v_n)$ is a formula
over ${\mathbb M}$ and 
$B=\{(x_1,...,x_m)\in {\mathbb M}^n: \phi(x_1,...,x_n) \text{ is true in } {\mathbb M}^n\}$
then
$^*B=\{(x_1,...,x_m) \in {^*{\mathbb M}}^n:
{^*{\phi}(x_1,...,x_n)} \text{ is true in } {^*{\mathbb M}}^n\}$,
where  $^*{\phi}(v_1,...,v_n)$ is the corresponding formula of  $\phi(v_1,...,v_n).$
\end{itemize}

The notion of a nonstandard extension and its properties can be generalized to
many sorted structures. This will be significant for the definition of the height function
which takes values in $\mathbb R.$
By a structure we mean a set equipped with some functions and relations on it.
For example, a ring is a structure with addition and multiplication.
A subset of a structure ${\mathbb M}$ which
is given by a first order formula is called a definable subset of ${\mathbb M}$.
We say that a structure ${\mathbb M}$ is $\aleph_1$-saturated if
whenever a collection of definable subsets
$(A_i)_{i \in I}$
whose parameters come from a countable set satisfies the finite intersection property
(that means for any finite subset $I_0$ of $I$ we have $\bigcap_{i \in I_0}A_i$ is not empty )
then $\bigcap_{i \in I}A_i$ is not empty.
\\
We assume all nonstandard extensions are $\aleph_1$-saturated.
Let
$$^*(K[X_1,...X_n])$$
be a proper nonstandard extension of $K[X_1,...X_n].$
For instance  an ultrapower of $K[X_1,...X_n]$
which respect to a nonprincipal ultrafilter on $\mathbb N$ is $\aleph_1$-saturated.
Ultraproducts of structures automatically become $\aleph_1$-saturated.
Note that $^*(R[X_1,...X_n])$, $^*R$ and $^*K$ are internal sets.
The height function $h$ on $R[X_1,...X_n]$  extends to
$^*(R[X_1,...X_n])$ which takes values in $^*{\mathbb R}$ though
this extension is no longer a height function if $h$ is unbounded.
Moreover it satisfies the same first order properties as $h$. In particular
if $x$, $y$ in $^*R$ with $h(x) \leq n$ and $h(y) \leq n$, where $n \in {^*{\mathbb N}}$,
then we have both $h(x+y) \leq \theta(n)$ and  $h(xy) \leq \theta(n).$
Note that  $^*K[X_1,...X_n] \subsetneq {^*(K[X_1,...X_n])}.$
Define
$$R_{fin}=\{x \in {^*{R}}: h(x) \in {\mathbb R}_{fin}\}$$
where ${\mathbb R}_{fin}=\{x \in {^*{\mathbb R}}: |x|<n \text{ for some } n \in {\mathbb N}\}$ and
${^*{\mathbb R}}$ is a nonstandard extension of ${\mathbb R}.$ The elements in
${^*{\mathbb R}} \setminus {\mathbb R}$ are called infinite.
\\

By the properties of the height function, if there is a height function on $R$,
we see that $R_{fin}$ is a subring of $^*R$ and it contains $R.$
The next lemma shows when $R_{fin}$ is internal.
\begin{lemma}
The set $R_{fin}$ is an internal subset of $^*R$ if and only if the height function on $R$ is bounded.
\end{lemma}
\begin{proof}
Suppose $R_{fin}={^*A}$ for some subset $A$ of $R$.
First we show that the height function on
$A$ must be bounded. To see this, if there is a sequence $(a_n)_n$ in $A$ such that
$\displaystyle\lim_{n \rightarrow \infty}h(a_n)=\infty,$ then by saturation there is an element in
$^*A$ whose height is infinite. This contradicts the fact that all the elements in
$R_{fin}$ have bounded height. So the height function on $A$ is bounded.
Therefore the height function on $^*A$ is also bounded. However since $R_{fin}$ contains $R$,
the height function on $R$ must be bounded.
Conversely if the height function on $R$ is bounded, then we have
$R_{fin}={^*R}$ and so $R_{fin}$ is internal.
\end{proof}
Now we fix some more notations.
Put $L=Frac(R_{fin})$ which is a subfield of $^*{K}$.
Note that $^*{K}$ is the fraction field of $^*{R}.$
Also we fix some algebraic closure $K^{alg}$ of $K.$
\\

For more detailed information about Nonstandard Analysis and Model Theory, the reader might 
consult \cite{Gold}, \cite{Hen} and \cite{Mar}.
In fact of being a height function is very related to the set $R_{fin}.$
The following proposition is the nonstandard point of view
definition of a height function. However it is ineffective, i.e. it does not provide
the $\theta$-type of the height function.

\begin{prop} \label{height}
A function $h:R \rightarrow [0,\infty)$ is a height function on $R$
if and only if $R_{fin}$ is a subring of $^*R.$
\end{prop}

\begin{proof}
We have seen that if $h$ is a height function then $R_{fin}$ is a subring.
Conversely suppose $R_{fin}$ is a subring and $h$ is not a height function.
This means there is some $N \in \mathbb N$ such that we have two sequences
$(r_n)$ and $(s_n)$ in $R$ with $h(r_n) \leq N$ and $h(s_n) \leq N$, but
$\displaystyle\lim_{n \rightarrow \infty}h(r_n \star s_n)=\infty$,
where the binary operation $\star$ means either addition or
multiplication.
By saturation, we get two elements $r$ and $s$ in $^*R$ such that
$h(r) \leq N$, $h(s) \leq N$ but $h(r \star s)$ is infinite. This contradicts the fact that
$R_{fin}$ is a subring.
\end{proof}
\subsection{Faithfulness and degree bounds}
In this subsection, we list some results from commutative algebra and
in particular about faithful extension of modules.
We refer the reader to  \cite{Bour}, \cite{Mat} or \cite{Mat2}.
Moreover we give the results in \cite{DS}
that lead to the existence
of the constant $c_1.$ 
\begin{lemma}\label{czero}
Let $F$ be a field and $f_1,...,f_s \in F[X_1,...X_n].$
Then $1 \in \langle f_1,...,f_s\rangle$ if and only if $f_1,...,f_s$ have no
common zeros in $F^{alg}.$

\end{lemma}

\begin{proof}
$\Rightarrow:$ Clear.
\\
$\Longleftarrow:$
By Hilbert's Nullstellensatz, there are $g_1,...,g_s \in {F^{alg}}[X_1,...,X_n]$
such that 
$1=f_1g_1+...+f_sg_s.$
This is a linear system when we consider the coefficients of all the polynomials.
Therefore $1=f_1Y_1+...+f_sY_s$ has a solution in $F^{alg}.$
Now by the Gauss-Jordan Theorem, this linear system has a solution in $F.$
So there are $h_1,...,h_s \in F[X_1,...,X_n]$
such that
$$1=f_1h_1+...+f_sh_s.$$

\end{proof}

\begin{definition} Let $A$ and $B$ be two rings and
$A \subseteq B$. We say that $B$ is a faithful extension of $A$,
if the ideal $BI$ is proper in $B$
whenever $I \subset A$ is a proper ideal.
\end{definition}

\begin{lemma} \label{faith}
Let $A$ and $B$ be two rings.
Suppose $A \subseteq B$ and $B$ is a faithful extension of $A$.
If $a,a_1,...,a_k$ are in $A$ and the linear equation
$$a_1x_1+...+a_kx_k=a$$
has a solution in $B$, then
it has a solution in $A.$
\end{lemma}

\begin{lemma} \label{faithful}
Let $F \subseteq F_1$ be a field extension. Then the extension
$F[X_1,...X_n]\subseteq F_1[X_1,...X_n]$ is faithful.
\end{lemma}

\begin{proof}
Let $I \subset F[X_1,...X_n]$ be a proper ideal.
Then since $I$ is finitely generated, $I=\langle f_1,...f_s \rangle$
for some $f_1,...,f_s \in F[X_1,...X_n].$
By \eqref{czero},  $f_1,...,f_s$ have a common zero in $F^{alg}.$ 
Since we may assume ${F}^{alg} \subseteq {F_1}^{alg}$,
there is a common zero of $f_1,...,f_s$ in ${F_1}^{alg}.$
So by \eqref{czero} again, $IF_1[X_1,...X_n] \neq F_1[X_1,...X_n].$
\end{proof}

\textbf{Fact:} The theory of algebraically closed fields is model complete.
\\
For more on this see \cite{Mar}.

\begin{lemma} \label{irreducible}
Let $F_1 \subset F_2$ be a field extension such that both are algebraically closed.
Let $V$ be an irreducible variety in ${F_1}^n$.
Then the Zariski closure of $V$ in ${F_2}^n$ (which respect to the Zariski topology on ${F_2}^n$)
is an irreducible variety in ${F_2}^n$.
\end{lemma}

\begin{proof}
Since $V$ is a variety in ${F_1}^n$, there are some polynomials $p_1,...,p_s$ such that
$V$ is the zero set of  $p_1,...,p_s$. Then clearly the Zariski closure
of $V$ in ${F_2}^n$ is the zero set of $p_1,...,p_s$ in ${F_2}^n$.
Call this closure $cl(V).$
Thus both $V$ and $cl(V)$ are defined by the formula
$$\phi(x)= \displaystyle\bigwedge_{i \leq s} p_i(x).$$
Now suppose that $cl(V)$ is not irreducible,
so there are two proper subvarieties $V_1$, $V_2$
of $cl(V)$ such that $cl(V)={V_1}\cup{V_2}.$
Then since the theory of algebraically closed fields is model complete,
we deduce that $V$ is reducible also.
\end{proof}

\begin{cor} \label{prime}
Let $F_1 \subset F_2$ be a field extension such that $F_1$ is
algebraically closed. Then
$I$ is a prime ideal in $F_1[X_1,...,X_n]$ iff $I{F_2}[X_1,...,X_n]$
is a prime ideal in ${F_2}[X_1,...,X_n].$

\end{cor}

\begin{proof}
Suppose $I=(f_1,...,f_s)$ is a prime ideal in $F_1[X_1,...,X_n].$
Let $V=V(I)$ be the variety given by $I.$
Then by Nullstellensatz $V$ is irreducible.
So by \eqref{irreducible}, the variety $cl(V)$ is also irreducible in ${F_2}^n$.
Again by Nullstellensatz, $I{F_2}[X_1,...,X_n]$ is prime.
The converse is true by \eqref{faithful} since
$(I{F_2}[X_1,...,X_n]) \cap F_1[X_1,...,X_n]=I.$
\end{proof}

For the following Lemma see  \cite[1.8]{DS}.

\begin{lemma} \label{Kfaith}
The extension $^*{K}[X_1,...X_n] \subset  {^*(K[X_1,...X_n])}$ is faithful.
\end{lemma}

Using \eqref{Kfaith}, we can obtain the existence of the constant $c_1.$
The original proof in \cite{DS} also uses the concept of flatness to prove the
existence of the constant $c_1$. For the details see  \cite[1.11]{DS}.
\begin{theorem} 
If $f_0,f_1,...,f_s$ in $K[X_1,...,X_n]$ all have degree less than $D$ and
$f_0$ is in $\langle f_1,...,f_s \rangle$, then
$f_0=\displaystyle\sum_{i=1}^{s}f_ih_i$ for certain ${h_i}$ whose degrees are bounded
by a constant $c_1=c_1(n,D)$ depending only on $n$ and $D$.

\end{theorem}

This is also from \cite[2.5]{DS}:

\begin{theorem} \label{sd-prime}
$I$ is a prime ideal in ${^*K}[X_1,...,X_n]$ iff $I{^*(K[X_1,...,X_n])}$
is a prime ideal in ${^*(K[X_1,...,X_n])}.$
\end{theorem}

\subsection{UFD with the p-property}
\begin{definition} We say that $R$ is a UFD with the p-property if
$R$ is an unique factorization domain endowed with an absolute value such that
every unit has absolute value 1 and if there are primes $p$ and $q$ satisfying
$$|p|<1<|q|,$$
then there is another prime $r$ non-associated to $p$
with $|r|<1.$

\end{definition}

\textbf{Examples}
\begin{itemize}
\item $\mathbb Z$ is a UFD with the p-property whose primes have absolute value bigger than 1.
\item $\mathbb Z_p$ ($p$-adic integers) is a UFD with the p-property whose only prime has
absolute value $1/p.$
\item Let $\gamma \in (0,1)$ be a transcendental number.
Then the ring $S={\mathbb Z}[\gamma]$ is a unique factorization domain since it is isomorphic
to ${\mathbb Z}[X]$ and its units are only 1 and -1. 
We put the usual absolute value on $S.$
Then $S$ has infinitely many primes $p$ with $|p|<1$ and 
infinitely many primes $q$ with $|q|>1.$
So $S$ is a UFD with the p-property.
\end{itemize}

\begin{lemma} \label{primes}
Suppose $R$ is a UFD with the p-property. If there are primes $p$ and $q$
with $|p|<1<|q|$,
then there are
infinitely many non-associated primes with absolute value strictly less than 1
and infinitely many non-associated primes with absolute value strictly bigger than 1.

\end{lemma}
\begin{proof}
We know there are at least two non-associated primes with absolute value less than 1.
Let $p_1,...,p_k$ (for $k \geq 2$) be non-associated primes with absolute value less than 1.
Put $A=p_1...p_k$. Now choose $m$ large enough such that
$\bigg|\displaystyle\sum_{i=1}^{k}(A/p_i)^m\bigg| < 1.$
Since this element is not a unit, it must be divisible by a prime whose absolute
value strictly less than 1.
This gives us a new prime. For the second part, given $q_1,...,q_k$
primes of absolute value larger than 1,
for large $n$  the element ${q_1}^nq_2...q_k + 1$
provides a new prime that has absolute value greater than 1.
\end{proof}

\section{Proof of Theorem A and Theorem B}
In this section we will give the proofs of Theorem A and Theorem B.
\begin{thA}

Let $R$ be a ring with a height function of $\theta$-type.
For all $n \geq 1$, $D \geq 1$, $H \geq 1$ there are two constants $c_1(n,D)$ and $c_2(n,D,H,\theta)$
such that if $f_1,...,f_s$ in $R[X_1,...X_n]$ have no common zero in $K^{alg}$ with
$\deg(f_i) \leq D$ and $h(f_i) \leq H$, then there exist nonzero $a$ in $R$ and
$h_1,...,h_s$ in $R[X_1,...X_n]$ such that
\begin{itemize}
\item[(i)] $a=f_1h_1+...+f_sh_s$
\item[(ii)] $\deg(h_i) \leq c_1$
\item[(iii)] $h(a),h(h_i) \leq c_2$
\item[(iv)] If $R$ is a UFD with the p-property
and $h(x)=|x|$ is the absolute value on $R$,
then we can choose $a$ such that $\gcd (a,a_1,...,a_m)=1$ where $a_1,...,a_m$ 
are all elements that occur as some coefficient of some $h_i.$
\end{itemize}
\end{thA}

\begin{remark}
The constant $c_1$ does not depend on $s$ because the vector space

$$
V(n,D)=\{f \in K[X_1,...X_n]: \deg(f) \leq D\}
$$
is  finite dimensional
over $K$. In fact the dimension is $q(n,D)={{n+D} \choose {n}}$.
Given $1=f_1h_1+...+f_sh_s$, we may always assume
$s \leq q=q(n,D)$ because if $s>q$ then $f_1,...,f_s \in V(n,D)$
are linearly dependent over $K$.
Assume first that $r \leq q$ many of them are linearly independent.
Therefore the other terms $f_{r+1},...,f_s$ can be written 
as a linear combination of $f_1,...,f_r$
over $K.$ Thus the equation $1=f_1h_1+...+f_sh_s$ may be transformed into another
equation $1=f_1g_1+...+f_rg_r.$ Consequently if $1 \in \langle f_1,...,f_s\rangle$,
then $1 \in \langle f_{i_1},...,f_{i_r}\rangle$ where $r \leq q$ and $i_j \in \{1,...,s\}.$
Hence, we can always assume $s=q.$ Similarly the constant $c_2$ does not depend on $s.$
Moreover, none of the constants depend on $R.$
\end{remark}

\begin{remark}
There is also a direct proof of Theorem A as follows:
Using the degree bound $B(n,D)$ 
for the polynomials
$g_1,...,g_s$ in a Bezout expression $1=f_1g_1+...+g_sf_s$, we can
derive a height bound since the degree bound allows
to translate the problem to solving a linear system of equations with precise number
of unknowns equations and the height function satisfies some additive and multiplicative properties.
However this computational method is also complicated 
since the bounds for the height function depend on $\theta$ which is implicitly given.
Thus in practice this method is ineffective.
For this reason and to show how the problem is related to Model Theory, 
we prefer nonstandard methods as in \cite{DS}.

\end{remark}

\begin{remark}
If  $R$ is a ring with absolute value which has arbitrarily small nonzero elements,
then we can multiply
both sides of the equation 
$$a=f_1h_1+...+f_sh_s$$
by some small $\epsilon \in R$.
Therefore the height bound $c_2$ can be taken 1
and the result becomes trivial.
Note that $(iv)$ in Theorem A prevents us from doing this
if there are no small units in $R.$
However  if there is a unit $u$ with $|u|<1$, then multiplying
both sides of the equation with powers of
$u$ the height can be made small again. So the interesting case is when 
there are no small units
which is equivalent to all the units having absolute value 1.
Note also that if $|ab| <1$ then $|a|$ can be
very big and $|b|$ can be very small.
So cancellation can make the height larger
if there are sufficiently small and big elements
in the ring. Thus for the equation
$$a=f_1h_1+...+f_sh_s,$$ simply
dividing by $\gcd(a,a_1,...,a_m)$ may not work 
in order to obtain  $(iv)$ in Theorem A.
\end{remark}

\textbf{Proof of Theorem A:}
If $s=1$ then by Nullstellensatz, $f_1$ must be a nonzero constant.
Thus we may assume that $s \geq 2$ and $f_1f_2$ is not 0.
By Theorem 2.9, the constant $c_1$ exists and it only depends on $n$ and $D.$
Now we prove the existence of the constant $c_2.$
Assume $n$, $D$ and $H$ are given and there is no bound $c_2.$
Therefore for every $m \geq 1$ there exists an integral domain
$R_m$ with a height function $ht_m$ of $\theta$-type
and $f_1,...,f_s$ in $R_m[X_1,...X_n]$ with $\deg f_i \leq D$ and $ht_m(f_i) \leq H$ witnessing to this.
Thus in the field of fractions $K_m$ of $R_m$,
there exist $g_1,...,g_s$ in $K_m[X_1,...,X_m]$ with $\deg g_i \leq c_1$ and 
$$1=f_1h_1+....+f_sg_s,$$
but for all $h_1,...,h_s \in K_m[X_1,...,X_n]$
with $\deg h_i \leq c_1$,
$$1=f_1h_1+...+f_sh_s$$ implies
$\displaystyle\max_j{ht_m(a_j)} > m$ where $a_j \in R_m$
is an element that occurs as a numerator
or denominator of some $h_i$.
Set
$$V_m(n,A)=\{f \in K_m[X_1,...,X_n]: \deg (f) \leq A\}$$ where $A \in \mathbb N.$
By Remark 3.1, this is a finite dimensional vector space over $K$ and the dimension is $q(n,A).$
So we can consider its elements as a finite tuple over $K_m$ and
any element of $V_m(n,A)$ is of the form $(a_1,...,a_{q(n,A)}).$

Also put $V_{R_m}(n,A):= V_m(n,A) \cap R_m[X_1,...,X_n].$
Note that the height of a tuple in $V_{R_m}(n,A)$ is the maximum of its coordinates.
Our language contains a symbol $h$ for the height function and
also the ring operations.
Consider the formula $\phi_m(v_1,...,v_s)$:
$$
\exists a^1_1\exists b^1_1 ... 
\exists a^1_{q(n,c_1)}\exists b^1_{q(n,c_1)}... 
\exists a^s_1\exists b^s_1 ... \exists a^s_{q(n,c_1)}\exists b^s_{q(n,c_1)}
$$

$$
\bigg(\bigg(\bigwedge_{i=1}^{s} (h(v_i)\leq H)\bigg) 
\wedge\bigg(1=\displaystyle\sum_{i=1}^{s}v_i\bigg(\frac {a^i_1}{b^i_1}
,...,\frac {a^i_{q(n,c_1)}}{b^i_{q(n,c_1)}}\bigg)\bigg) \bigg)
$$

$$
\wedge\bigg(\forall d^1_1\forall e^1_1 ...\forall d^1_{q(n,c_1)}\forall e^1_{q(n,c_1)}
...\forall d^s_1\forall e^s_1 ...\forall d^s_{q(n,c_1)}\forall e^s_{q(n,c_1)}
$$

$$
\bigg(\bigg(\bigwedge_{i}(\max_j(h(d^i_j),h(e^i_j)) \leq m)\bigg)
\rightarrow
 \bigg(1 \neq \displaystyle\sum_{i=1}^{s}v_i\bigg(\frac{d^i_1}{e^i_1},...,
\frac{d^i_{q(n,c_1)}}{e^i_{q(n,c_1)}}\bigg)\bigg) \bigg).
$$
Note that this formula can be seen as a formula in $R_m$
by seeing each $v_i$ as the tuple of variables representing the polynomial $f_i$.
We see that $R_m \models \psi_m$ where
$\psi_m=\exists v_1...\exists v_s \phi_m(v_1,...,v_s).$
By compactness there is an integral domain $R$ with a height function
$h_R$ of $\theta$-type that satisfies all $\psi_m.$
Now we consider the set of formulas
$$p(v_1,...,v_s)=\{\phi_m(v_1,...,v_s): m=1,2,3...\}.$$
This is a set of formulas over $^*V_R(n,D)$ using countably many parameters.
This set is finitely consistent, so by saturation there is a realization
$f_1, ..., f_s$ in $^*(R[X_1,...,X_n])$.
But according to $p(v_1,...,v_s)$, the polynomials $f_1, ..., f_s$ are in $R_{fin}[X_1,...,X_n]$
and their degrees are less than $D$.
Furthermore the linear system
$$f_1Y_1+...+f_sY_s = 1$$
has a solution in $^*K[X_1,...,X_n]$ 
(because bounded degree polynomials of $^*(K[X_1,...,X_n])$ 
are in $^*K[X_1,...,X_n]$) but not in $L[X_1,...,X_n]$.
This contradicts \eqref{faith} because the extension 
$L[X_1,...,X_n] \subset {^*K[X_1,...,X_n]}$ is faithful by \eqref{faithful}.

Hence we know that given $f_1,...f_s \in R[X_1,..,X_n]$
with no common zeros in $K^{alg}$ with $\deg(f_i) \leq D$ and $h(f_i) \leq H$,
there are $h_1,...,h_s$ in $K[X_1,...X_n]$ such that
$1=f_1h_1+...+f_sh_s$ and $\deg(h_i) \leq c_1(n,D).$ Moreover $s \leq q(n,D)$ and
$h(e) \leq c_3(n,D,H,\theta)$ where $e \in R$ is an element
which occurs as a numerator or denominator for some coefficient of some $h_i.$
Let $b_1,...,b_m$ be all the elements in $R$
that occur as a denominator for some coefficient of some $h_i.$
Note that $m=m(n,D) \leq q^2$ depends on $n$ and $D$ only.
Also we know that $h(b_i) \leq c_3.$ Put
$$a=b_1...b_m.$$
By the multiplicative properties of the height function,
we get $h(a) \leq c_4(n,D,H,\theta)$ for some $c_4.$
Now we see that
$$a=\displaystyle\sum_{i=1}^{s}f_i(ah_i),$$
$f_i$ and $ah_i$ are in $R[X_1,..,X_n]$ and $\deg(ah_i)=\deg(h_i) \leq c_1.$
Moreover, again by the multiplicative properties of the height function,
we have $h(ah_i) \leq c_5(n,D,H,\theta).$ Now take $c_2=\max(c_4,c_5).$
Therefore we obtain $(i)$, $(ii)$ and $(iii).$
\\
Now we prove $(iv)$.
Assume $R$ is a UFD with the p-property.
We need to choose $a$ such that
$\gcd (a,a_1,...,a_m)=1$ where $a_1,...,a_m$
are all elements that occur as some coefficient of some $h_i.$
If all the primes in $R$ have absolute value bigger than 1 or smaller than 1, then
we can divide both sides of the equation 
$$a=f_1h_1+f_2h_2+...+f_sh_s$$
by $\gcd(a,a_1,...,a_m)$ and get the result because
if all the primes in $R$ have absolute value bigger than 1, then
cancellation makes the height smaller and if
all the primes in $R$ have absolute value less than 1
then height is bounded by 1.
The remaining case is when there are primes of absolute value bigger than 1
and primes of absolute value smaller than 1.
By \eqref{primes}, there are infinitely many primes with absolute value
strictly less than 1. Now choose a prime $p$ such that $|p|<1$ and $p$ does
not divide $a.$
Let $d$ be the greatest common divisor of all coefficients of $f_1$ and $f_2.$
Then, the coefficients of $f_1/d$ and $f_2/d$ have no common divisor.
On the other hand, since there are both small and large
elements in the ring, $d$ can be very small and so
$f_1/d$ and $f_2/d$ may have very large absolute values.
Thus choose a natural number $k$ such that $p^kf_1/d$ and $p^kf_2/d$ have absolute value
less than 1. Put $v=c_1(n,D)+1.$
Then we have
$$0=f_1({X_1}^vp^kf_2/d)+f_2(-{X_1}^vp^kf_1/d).$$
Therefore we obtain that
$$a=f_1(h_1+{X_1}^vp^kf_2/d)+f_2(h_2 -{X_1}^vp^kf_1/d )+...+f_sh_s$$

$$=f_1g_1+f_2g_2+...+f_sg_s$$
where $\deg g_i \leq D(c_1 +1)=c(n,D)$ and $h(g_i)\leq c_2$.
Observe that 
$$\gcd (a,a_1,...,a_m)=1$$
where $a_1,...,a_m$
are all elements that occur as some coefficient of some $g_i.$
\qed

Next we prove Theorem B.
\\
From now on, K denotes the algebraic numbers $\overline{\mathbb Q}$ and $^*K$ its nonstandard extension.
Set $$L=K_{fin}=\{x \in {^*{K}}: h(x) \in {\mathbb R}_{fin}\}.$$
Before proving Theorem B we need one more lemma.

\begin{lemma} \label{alg}
L is an algebraically closed field.
\end{lemma}

\begin{proof}
Since the logarithmic height function behaves under algebraic operations and inverse, $L$ is a field.
By height inequality \eqref{heightt} we see that $L$ is algebraically closed.
\end{proof}

\textbf{Proof of Theorem B:}

First note that if $J=(f_1,...,f_s)$ is an ideal of $D$-type then the number
of generators of $J$ can be taken less than
$$q=q(n,D)=\dim_{K}\{f \in K[X_1,...,X_n]: \deg f \leq D \}.$$
So we can always assume $s \leq q.$
We know the existence of the bound B=$B(n,D)$ by  \cite{DS}.
Now we prove the existence of the bound $C(n,D,H).$
Suppose there is no such a bound. This means for all $m>0$
there is an ideal $I_m$ of $(D,H)$-type of $\overline{\mathbb Q}[X_1,...,X_n]$ which is not prime
such that for all $f,g$ with $\deg f$, $\deg g$ less than $B$ and
$h(f)$, $h(g)$ less than $m$, $fg\in I$ implies $f$ or $g$ in $I$.
Then by compactness there is an ideal $I$ of $(D,H)$-type of 
$^*(\overline{\mathbb Q}[X_1,...,X_n])$ such that I is not prime but for all
$m>0$ if $f,g$ are of degree less than $B$ and are of height 
less than $m$, $fg\in I$ implies $f$ or $g$ in $I$.
Now we see that the ideal $I$ is prime in $L[X_1,...,X_n].$
However it is not prime in ${^*K}[X_1,...,X_n]$ by \eqref{sd-prime}.
This contradicts to \eqref{prime} since $L$ is algebraically closed by \eqref{alg}.
\\

\textbf{Question:} Can we compute $C(n,D,H)$ in Theorem B effectively?

\section{Further Results}
In this section we prove some consequences on Theorem B.
For details we refer the reader to \cite{Bour, Eisenbud}.
First recall the followings:
\begin{itemize}

\item An ideal $J$ is a primary ideal if and only if
$Ass_{R}(R/J)$ has exactly one element.
\item Every ideal $J$ (through primary decomposition)
is expressible as a finite intersection of primary ideals.
The radical of each of these ideals is a prime ideal
and these primes are exactly the elements of $Ass_{R}(R/J)$ .
\item Any prime ideal minimal with respect to containing an ideal $J$ is in $Ass_{R}(R/J)$.
These primes are precisely the isolated primes.

\end{itemize}

\begin{cor} Let $n \in \mathbb N$, $X=(X_1,...,X_n)$, $I$ be an ideal of $L[X].$

\begin{itemize}

\item [(1)]If $p_k,...p_m$ are the distinct minimal primes of $I$ then
$${p_1}{^*K}[X],... ,{p_m}{^*K}[X]$$
are the distinct minimal primes of $I{^*K}[X_1,...,X_n].$
\item [(2)] $\sqrt{I{^*K}[X]}={\sqrt{I}}{^*K}[X].$
\item [(3)] If $M$ is an $L[X]$-module, then
$$Ass_{{^*K}[X]}(M \otimes_{L[X]} {^*K}[X])=\{p{^*K}[X]: p \in Ass_{L[X]}(M)\}.$$
\item [(4)]$I$ is primary ideal iff $I{^*K}[X]$ is primary ideal of ${^*K}[X].$
\item [(5)] Let $I=I_1 \cap...\cap I_m$ be a reduced primary decomposition, $I_k$ being a
$p_k$-primary ideal.
Then $$I{^*K}[X]=I_1{^*K}[X]\cap...\cap I_m{^*K}[X]$$
is a reduced primary decomposition of $I{^*K}[X],$ and
$I_k{^*K}[X]$ is a $p_k{^*K}[X]$-primary ideal.

\end{itemize}

\end{cor}

\begin{proof}
(1) is an immediate consequence of Theorem B.
(2) follows from (1) since radical of an ideal is the intersection of minimal prime ideals
which contain the ideal. Since $L[X]$ is noetherian,
(3) follows from \cite[Chapter 4, 2.6, Theorem 2]{Bour}  and \eqref{faithful}.
To prove (4), suppose that $I$ is a $p$-primary ideal.
So we get $Ass_{L[X]}(L[X]/I)=\{p\}.$
Applying (3) with $M=L[X]/I$ we obtain that
$Ass_{{^*K}[X]}({^*K}[X]/I)=\{p{^*K}[X]\}.$
This proves (4). The converse of (4) is true by \eqref{faithful}.
(5) follows from (4).

\end{proof}

Now we give the standard corollaries.
For the following corollary, the existence of the constant 
$E(n,D,H)$ is new.

\begin{cor}
There are constants $B(n,D)$, $C(n,D)$ and $E(n,D,H)$ such that
if $I$ is an ideal of $(D,H)$-type, then
\begin{itemize}

\item [(1)]$\sqrt{I}$ is generated by polynomials of degree less than $B$ and 
height less than $E$, if $f \in \sqrt{I}$ then $f^C \in I.$
\item [(2)]There are at most $B$ associated primes of $I$ and each
generated by polynomials of degree less than $B$ and height less than $E.$
\item[(3)] $I$ is primary iff $1 \notin I$, and for all $f,g$ of degree less than $B$
and height less than $E$, if $fg \in I$ then $f \in I$ or $g^C \in I.$
\item[(4)] There is a reduced primary decomposition of $I$ consisting of at most $B$
primary ideals, each of which is generated by polynomials of degree at most $B$
and height at most $E.$

\end{itemize}

\end{cor}

\begin{proof}
We know the existence of $B(n,D)$ and $C(n,D)$ by \cite{DS}.
The existence of $E(n,D,H)$
follows from the previous corollary.
Proofs are similar to the proof of Theorem B.
Details are left to the reader.
\end{proof}

\textbf{Question:} Can we compute $E(n,D,H)$ effectively in Corollary 4.2?

\section{Concluding Remarks and Further Discussion of Theorem A}
In this section we discuss the Theorem A in terms of unique factorization domains,
valuations and some arithmetical functions.
Also we give some counter examples for the Theorem A for non-height functions.

\subsection{UFD with the 1-property}
\begin{definition}
We say that $R$ is a UFD with the 1-property if
$R$ is an unique factorization domain endowed with an absolute value such that
every unit has absolute value 1 and there is only one prime $p$ of
absolute value less than 1 and infinitely many primes $q$ of absolute value greater than 1.
\end{definition}
\
\textbf{Example:}
Let $R$ be an unique factorization domain and
$p$ be a prime in $R$.
Put the $p$-adic absolute value on $R$ with $|p|_p=1/2.$
Let $c>1$ be any real number.
On $R[X]$ we define
$$|a_0+a_1X+...+a_kX^k|=\max_{i}c^i|a_i|_p. $$
Then $R[X]$ is a UFD with the 1-property whose only small prime is $p.$
\\

We proved the Theorem A for UFD with the p-property.
Thus the remaining case is when $R$ is a UFD with the 1-property.
Now we show the Theorem A is not true for a UFD with the 1-property.
The reason behind this is the fact that an element has small absolute value if and
only if its $p$-adic valuation is very large where $p$ is the unique prime of
absolute value less than 1.

\begin{prop}
Let $R$ be a UFD with the 1-property.
Then we cannot ensure the correctness of $(iii)$ and $(iv)$ simultaneously in Theorem A.
\end{prop}

\begin{proof}
Let $p$ be the unique small prime in $R$ of absolute value less than 1.
Let $B$ be an element in $R$ of absolute value very big which is coprime to $p.$
Choose $m$ minimal such that $|p^mB| \leq 1$.
Similarly choose $k$ minimal such that $|p^kB| \leq c_2$.
Note that as B is very large then so are $m$ and $k.$
Set $f_1=p^{2m+1}+p^{2m}X$ and $f_2=p^{m}B-p^{m}BX.$
Clearly $f_1$ and $f_2$ have no common zero
since
$$p^{2m}B(p+1)=Bf_1 +p^mf_2$$
and $p$ is not -1.
Whenever we write $a=f_1h_1+f_2h_2$, we get that $p^m$ divides $h_2$ and $B$ divides $h_1.$
Also we have that $p^{2m}B$ divides $a.$
Now suppose $|h_i| \leq c_2$ for $i=1,2.$
Since $B$ divides $h_1$, we see that $p^k$ divides $h_1$ since $p$ is the unique small prime in $R.$
Thus $p^k$ divides $a$, $h_1$ and $h_2.$
Furthermore we may assume that the only prime divisor of $a$, $h_1$ and $h_2$ is
$p$, because if there is $q$ dividing all of them
which is coprime to $p$, then there is $l \geq k$ such that $p^l$ divides $h_1$
in order to make the absolute value of $h_1$ less than $c_2.$
Similar observation shows that $p^l$ also divides $h_2$ and $a.$
Therefore, in order to satisfy $(iv)$ in Theorem A, we need to divide
$a$, $h_1$ and $h_2$ by $p^k.$ So the absolute value of $h_1/p^k$
becomes very large.

\end{proof}

\subsection{Valuations}
\begin{definition}
A valuation $v$ on an integral domain $R$ is a function
$v:R\rightarrow \Gamma \cup \{\infty \}$ from $R$ into an ordered abelian group $\Gamma$ 
that satisfies the followings:
\begin{itemize}
\item[(i)] $v(a)=\infty$ if and only if $a=0$
\item[(ii)] $v(xy)=v(x)+v(y)$ 
\item[(iii)] $v(x+y) \geq \min(v(x),v(y)).$
\end{itemize}
Here $\infty$ is some element that is bigger than every element in $\Gamma.$
\end{definition}
For a nonzero polynomial in $n$-variable we define its valuation as follows:
$$v\bigg(\displaystyle\sum_{\alpha}a_{\alpha}X^{\alpha}\bigg)
=\displaystyle\max_{\alpha}\{v(a_{\alpha}): a_{\alpha} \neq 0 \}.$$
Note that this may not be a valuation that satisfies the three conditions above.
Take $R=\mathbb Z$ and as a valuation we put a $p$-adic valuation for some prime $p.$
Set $f_1=1+X+(1-p^m)X^2$ and $f_2=X^3$ where $m$ is some large integer.
Then the valuations of $f_1$ and $f_2$ are 0 and
clearly they have no common zero in $\mathbb C$.
One can see that 1 is a linear combination of
$f_1$ and $f_2$ and so every integer is.
However, whenever we write $a=f_1h_1+f_2h_2$ where $a$ is nonzero,
then $h_1$ must have degree bigger than 2 and the first three coefficients of $h_1$ are
uniquely determined: if $h_1(x)=b_0+b_1X+b_2X^2+...+b_kX^k$ then
automatically we have $b_0=a$, $b_1=-a$ and $b_2=ap^m.$
So the valuation of $b_2$ can be very large.
The main nonstandard reason behind this is the fact that
$$R_{vfin}=\{x \in {^*R}: v(x) \in {\mathbb R}_{fin}\} \cup \{0\}$$
is not a ring, because
for nonstandard $N \in {^*{\mathbb N}}$ the elements $p^N - 1$ and 1 is in $R_{vfin}$ but
not their sum. Therefore by \eqref{height}, we know that the $p$-adic valuation on
$\mathbb Z$ is not a height function.
\\

If we take $g_1=p^m-1 + X$ and $g_2=1-X$ then they have no common zero and 
whenever we write $a=g_1h_1+g_2h_2$, then $h_1$ and $h_2$ must have the same degree and
same leading coefficient. This implies that $p^m$ divides $a$ which means that valuation
of $a$ can be very big even if the valuations of $g_1$ and $g_2$ are 0.
\\

A valuation is called trivial if for all nonzero $x$ we have $v(x)=0.$
We say that a valuation is a height function if the set
$R_{vfin}$ is a subring.
In fact we can determine when a valuation is a height function.

\begin{lemma}
A valuation $v$ on $R$ is a height function if and only if it is trivial.
\end{lemma}
\begin{proof}
If the valuation is trivial then clearly it is a height function.
Conversely is $v$ is not trivial, then it is unbounded.
So by saturation there is an element $a$ in $^*R$ whose valuation is infinite.
Then
$$v(a-1)=0$$
because if two elements have different valuation then the valuation of their sum
is the minimum of their valuations. So the elements $a-1$ and 1 are in $R_{vfin}$,
but not their sum.
\end{proof}

\subsection{Arithmetical Functions}
Now we discuss some arithmetical functions and which of them are height functions.
\begin{definition}

A function $g: \{1,2,3,...\} \rightarrow \mathbb C$ is called an arithmetical function.
\end{definition}
Every arithmetical function $g$ extends to $\mathbb Z$ by defining $g(n)=g(-n)$ and $g(0)=0.$
Such a function on $\mathbb Z$ is called an arithmetical function on $\mathbb Z.$
Similarly for an  an arithmetical function $g$ on $\mathbb Z$, we extend it to ${\mathbb Z}[X]$ by
$$g(a_0+a_1X+...+a_kX^k)=\max_{i}g(a_i).$$
Let $^*{\mathbb Z}$ be a proper nonstandard extension of $\mathbb Z.$
Note that
$${\mathbb Z}_{fin}=\{x \in {^*{\mathbb Z}}: |x|<n \text{ for some } n \in {\mathbb N}\}= \mathbb Z.$$
For an arithmetical function $g$, we define
$${\mathbb Z}_{gfin}=\{x \in {^*{\mathbb Z}}: |g(x)|<n \text{ for some } n \in {\mathbb N}\}.$$
By \eqref{height}, $|g|$ is a height function if and only if ${\mathbb Z}_{gfin}$ is a subring.
Now we give some examples of arithmetical functions.
\\
\textbf{Examples: }
\begin{itemize}
\item $\varphi(n)= |\{1\leq k \leq n: (k,n)=1\}|$
\item $\pi(n)= \text{ number of primes less than n}$
\item $d(n)=\text{ number of divisors of n}$
\item $\omega(n)=\text{ number of distinct prime factors of n}.$
\end{itemize}

\begin{lemma}
Let $g$ be an arithmetical function and assume 
that
$$\lim_{n \rightarrow \infty}g(n)= \infty.$$
Then $|g|$ is a height function.
\end{lemma}

\begin{proof}
If $N$ is an infinite number in $^*{\mathbb Z}$ then $g(N)$ is also infinite.
This shows that ${\mathbb Z}_{gfin}={\mathbb Z}_{fin}=\mathbb Z$ which is a subring of $^*{\mathbb Z}.$
Hence by \eqref{height}, $|g|$ is a height function on $\mathbb Z.$
\end{proof}

\begin{lemma}
For all $n \geq 1$, we have $\frac{\sqrt{n}}{2} \leq \varphi(n).$
\end{lemma}

\begin{proof}
Since $\varphi(n) = \displaystyle\prod_{p|n}n(1-\frac{1}{p})$, we get
$\varphi(n) \geq \frac{n}{2^{\omega(n)}} \geq \frac{n}{d(n)}.$
Finally since $d(n) \leq 2\sqrt{n}$, we get the result.
\end{proof}
\begin{cor}
The functions $\pi(n)$ and $\varphi(n)$ are height functions.
\end{cor}

\begin{proof}
Since there are infinitely many primes and $\frac{\sqrt{n}}{2} \leq \varphi(n)$, these two functions
are height functions.
\end{proof}

For the other two functions $d(n)$ and $\omega(n)$, they take small values when $n$ is a prime number.
\\

\textbf{Fact:} Every sufficiently large odd integer can be written
as a sum of three primes. This was proved by I. M. Vinogradov.
For more about this theorem, we refer the reader to \cite{Dav}.

\begin{lemma}
The functions $d(n)$ and $\omega(n)$ are not height functions.
\end{lemma}

\begin{proof}
By three primes theorem and the transfer formula,
there is an odd infinite $N$ in $^*{\mathbb Z}$ which
can be written as a sum of three primes in $^*{\mathbb P}$ where $\mathbb P$
is the set of all primes.
Furthermore we can choose $N$ such that
$\omega(N)$ is infinite.
This shows that the sets ${\mathbb Z}_{\omega{fin}}$ and ${\mathbb Z}_{dfin}$
are not closed under addition. So by \eqref{height}, they cannot be height functions on $\mathbb Z.$
\end{proof}

The next two Corollaries are also true for the function  $\omega(n).$
For simplicity, we just give the proofs for the divisor function.
\begin{cor} \label{divisor}
There exist a natural number $A$ and two sequences $\{a_n\}$ and $\{b_n\}$
in $\mathbb N$ such that
$d(a_n) \leq A$ and $d(b_n) \leq A$ but 
$$\lim_{n \rightarrow \infty}{d(a_n + b_n)}=\infty.$$
\end{cor}

\begin{cor}
Theorem A is not true for the function $d(n).$
\end{cor}

\begin{proof}
Set $f_1=a_n+X+{{b_n}^2}X^2$ and $f_2=X^3$ where $a_n$ and $b_n$ are as in \eqref{divisor}.
Then $d(f_1)$ and $d(f_2)$ are bounden by $A^2$ and
they have no common zero in $\mathbb C.$
However, whenever we write $a=f_1h_1+f_2h_2$ where $a$ is nonzero,
then $h_1$ must have degree bigger than 2 and the first three coefficients of $h_1$ are
uniquely determined: if $h_1(x)=c_0+c_1X+c_2X^2+...+c_kX^k$ then
automatically we have $c_0=a$, $c_1=-a_na$ and $c_2=a(a_n - b_n)(a_n+b_n).$
Hence $d(c_2)$ can be very large.
Moreover if we put $g_1=a_n + X$ and $g_2=b_n-X$ then they have no common zero.
However, whenever
we write  $a=g_1h_1+g_2h_2$, then $d(a) \geq d(a_n+b_n).$
Thus $a$ has many divisors although $d(g_1)$ and $d(g_2)$ are bounded by $A.$

\end{proof}

\textbf{Acknowledgements.} The author thanks Amador Martin-Pizarro and Frank Wagner for very fruitful discussions related to this paper.

\end{document}